\documentclass[12pt]{amsart}
\usepackage{amsmath,amssymb,amsbsy,amsfonts,latexsym,amsopn,amstext,
                                               amsxtra,euscript,amscd,bm}
\usepackage{comment}
\usepackage{units}
\usepackage{float}
\restylefloat{table}
\usepackage[english]{babel}                    
\usepackage{url}
\usepackage[colorlinks,linkcolor=blue,anchorcolor=blue,citecolor=blue]{hyperref}
\usepackage{color}
\usepackage[utf8]{inputenc}

\usepackage{listings}
\usepackage{xcolor}

\usepackage{blindtext}

\definecolor{codegreen}{rgb}{0,0.6,0}
\definecolor{codegray}{rgb}{0.5,0.5,0.5}
\definecolor{codepurple}{rgb}{0.58,0,0.82}
\definecolor{backcolour}{rgb}{0.95,0.95,0.92}

\lstdefinestyle{mystyle}{
    backgroundcolor=\color{backcolour},   
    commentstyle=\color{codegreen},
    keywordstyle=\color{magenta},
    numberstyle=\tiny\color{codegray},
    stringstyle=\color{codepurple},
    basicstyle=\ttfamily\footnotesize,
    breakatwhitespace=false,         
    breaklines=true,                 
    captionpos=b,                    
    keepspaces=true,                 
    numbers=left,                    
    numbersep=5pt,                  
    showspaces=false,                
    showstringspaces=false,
    showtabs=false,                  
    tabsize=2
}

\begin{document}

\newtheorem*{note}{Note}
\newtheorem{problem}{Problem}
\newtheorem{theorem}{Theorem}
\newtheorem{lemma}[theorem]{Lemma}
\newtheorem{claim}[theorem]{Claim}
\newtheorem{cor}[theorem]{Corollary}
\newtheorem{prop}[theorem]{Proposition}
\newtheorem{definition}{Definition}
\newtheorem{question}[theorem]{Question}
\newtheorem{conjecture}{Conjecture}
\def\cA{{\mathcal A}}
\def\cB{{\mathcal B}}
\def\cC{{\mathcal C}}
\def\cD{{\mathcal D}}
\def\cE{{\mathcal E}}
\def\cF{{\mathcal F}}
\def\cG{{\mathcal G}}
\def\cH{{\mathcal H}}
\def\cI{{\mathcal I}}
\def\cJ{{\mathcal J}}
\def\cK{{\mathcal K}}
\def\cL{{\mathcal L}}
\def\cM{{\mathcal M}}
\def\cN{{\mathcal N}}
\def\cO{{\mathcal O}}
\def\cP{{\mathcal P}}
\def\cQ{{\mathcal Q}}
\def\cR{{\mathcal R}}
\def\cS{{\mathcal S}}
\def\cT{{\mathcal T}}
\def\cU{{\mathcal U}}
\def\cV{{\mathcal V}}
\def\cW{{\mathcal W}}
\def\cX{{\mathcal X}}
\def\cY{{\mathcal Y}}
\def\cZ{{\mathcal Z}}

\def\A{{\mathbb A}}
\def\B{{\mathbb B}}
\def\C{{\mathbb C}}
\def\D{{\mathbb D}}
\def\E{{\mathbb E}}
\def\F{{\mathbb F}}
\def\G{{\mathbb G}}
\def\I{{\mathbb I}}
\def\J{{\mathbb J}}
\def\K{{\mathbb K}}
\def\L{{\mathbb L}}
\def\M{{\mathbb M}}
\def\N{{\mathbb N}}
\def\O{{\mathbb O}}
\def\P{{\mathbb P}}
\def\Q{{\mathbb Q}}
\def\R{{\mathbb R}}
\def\S{{\mathbb S}}
\def\T{{\mathbb T}}
\def\U{{\mathbb U}}
\def\V{{\mathbb V}}
\def\W{{\mathbb W}}
\def\X{{\mathbb X}}
\def\Y{{\mathbb Y}}
\def\Z{{\mathbb Z}}

\def\ep{{\mathbf{e}}_p}
\def\em{{\mathbf{e}}_m}
\def\eq{{\mathbf{e}}_q}

\def\scr{\scriptstyle}
\def\\{\cr}
\def\({\left(}
\def\){\right)}
\def\[{\left[}
\def\]{\right]}
\def\<{\langle}
\def\>{\rangle}
\def\fl#1{\left\lfloor#1\right\rfloor}
\def\rf#1{\left\lceil#1\right\rceil}
\def\le{\leqslant}
\def\ge{\geqslant}
\def\eps{\varepsilon}
\def\mand{\qquad\mbox{and}\qquad}

\def\sssum{\mathop{\sum\ \sum\ \sum}}
\def\ssum{\mathop{\sum\, \sum}}
\def\ssumw{\mathop{\sum\qquad \sum}}

\def\vec#1{\mathbf{#1}}
\def\inv#1{\overline{#1}}
\def\num#1{\mathrm{num}(#1)}
\def\dist{\mathrm{dist}}

\def\fA{{\mathfrak A}}
\def\fB{{\mathfrak B}}
\def\fC{{\mathfrak C}}
\def\fU{{\mathfrak U}}
\def\fV{{\mathfrak V}}

\newcommand{\bflambda}{{\boldsymbol{\lambda}}}
\newcommand{\bfxi}{{\boldsymbol{\xi}}}
\newcommand{\bfrho}{{\boldsymbol{\rho}}}
\newcommand{\bfnu}{{\boldsymbol{\nu}}}

\def\GL{\mathrm{GL}}
\def\SL{\mathrm{SL}}

\def\Hba{\overline{\cH}_{a,m}}
\def\Hta{\widetilde{\cH}_{a,m}}
\def\Hb1{\overline{\cH}_{m}}
\def\Ht1{\widetilde{\cH}_{m}}

\def\flp#1{{\left\langle#1\right\rangle}_p}
\def\flm#1{{\left\langle#1\right\rangle}_m}
\def\dmod#1#2{\left\|#1\right\|_{#2}}
\def\dmodq#1{\left\|#1\right\|_q}

\def\Zm{\Z/m\Z}

\def\Err{{\mathbf{E}}}

\newcommand{\comm}[1]{\marginpar{%
\vskip-\baselineskip 
\raggedright\footnotesize
\itshape\hrule\smallskip#1\par\smallskip\hrule}}

\def\xxx{\vskip5pt\hrule\vskip5pt}

\newenvironment{nouppercase}{%
  \let\uppercase\relax%
  \renewcommand{\uppercasenonmath}[1]{}}{}
  

\title{Approximated solution of a
differential-difference equation arising in number theory and applications to the linear sieve}
\author{Matteo Bordignon\\ University of New South Wales Canberra, School of Science \\ m.bordignon@student.unsw.edu.au }

\date{\today
}


\begin{abstract}
We provide elementary and accurate numerical solutions to the differential-difference equation, which improves an explicit version of the linear sieve given by Nathanson.
 \end{abstract}
\begin{nouppercase}
\maketitle
\end{nouppercase}
\section{Introduction}
This article focuses on methods to approximate the solution $f_n(s)$, for $n\ge 1$ and $s \ge 1$, of the following differential-difference equation
\begin{equation}
\label{eq:diffeq}
\begin{cases} 
      f_n(s)=0 \quad \text{for} \quad s \ge n+2 \\
      sf_1(s)=3-s \quad \text{for} \quad 1 \le s \le 3\\
      (sf_n(s))'=-f_{n-1}(s-1) \mkern6mu \text{for} \mkern6mu n\ge 2 \mkern6mu\text{even} \mkern6mu \& \mkern6mu s\ge 2 \mkern6mu \\
      (sf_n(s))'=-f_{n-1}(s-1) \mkern6mu \text{for} \mkern6mu n\ge 2 \mkern6mu \text{odd} \mkern6mu \& \mkern6mu s\ge 3\\
      (sf_n(s))'=0 \mkern6mu \text{for}\mkern6mu n\ge 2 \mkern6mu \text{odd} \mkern6mu \& \mkern6mu 1 \le s\le 3.
      \end{cases} 
\end{equation}
This differential-difference equation is classical in analytic number theory as it is deeply related to the famous linear sieve, which proves Chen's theorem \cite{Chen2, Chen1, Chen}. 
\begin{theorem}[Chen]
\label{theo:Chen}
All sufficiently large even integers can be written as a sum of a prime and a semi-prime.
\end{theorem}
Given that the field of differential-difference equation is a prolific one, there are many results that can be found on similar equations, see the works of Bellman and Kotkin \cite{Bellman}, van de Lune and Wattel \cite{Van}, Lal and Gillard \cite{Lal},  G. Marsaglia, Zaman and J. Marsaglia \cite{Marsaglia},  Wheeler \cite{Wheeler1, Wheeler2}, Moree \cite{Moree}, Bach \cite{Bach}, Sorenson \cite{Sorenson} and Bradley \cite{Bradley}. Despite this and \eqref{eq:diffeq} being well known, there are relatively few works in the literature giving approximate numerical solutions for this specific diffrential-difference equation. The only results of this kind are in Chapter 9 of the classic book by Nathanson \cite{Nathanson}, building upon unpublished lecture notes by Iwaniec \cite{Iwaniec}. It is interesting to note that in this work the approach is mainly analytic. \hfill \break 
We will focus on improving Nathanson's result, combining his analytic approach with a more compuational one. In doing so we draw inspiration from the similarities between the above differential-difference equation and Dickman function \cite{Dickman}, to obtain a good numerical solution. We will adapt the numerical techniques developed by G. Marsaglia, Zaman and J. Marsaglia \cite{Marsaglia}. Specifically, we will turn the differential-difference equation in an integral-difference equation and solve it via Taylor expansions. This method allows for a rapid rate of convergence to the solution.\hfill \break
We also  introduce an elementary method, which is based on approximating the integral of a decreasing function with a weighted sum, to obtain an upper bound for the function $f_n(s)$ for `small' $n$. The chosen upper bound function is the following
\begin{equation*}
h(s)= \begin{cases} 
      e^{-2} & 1\le s \le 2 \\
      e^{-s} & 2\le s \le 3 \\
      3s^{-1}e^{-s} & s\ge 3,
      \end{cases} 
\end{equation*}
which was chosen by Nathanson and indeed it appears to be a numerically good approximation for $f_n(s)$. Here our purpose is to approximate $c_n$, the smallest number such that if $n$ is odd and $s\ge 1$, or if $n$ is even and $s\ge 2$, then
\begin{equation}
\label{eq:fh}
f_n(s)\le 2e^2 (c_n)^{n-1}h(s).
\end{equation}
It is interesting to note that a result of this kind is useful in applications as will appear clear in Section \ref{ELS}. On the other hand, changing sightly our approach, it is surely possible to obtain a non-uniform upper bound that will be better in certain ranges, but we chose not to pursue this as the above result appears good enough for our applications.
The computational method introduced in Subsection \ref{NUB} will lead to the following result.
\newpage
\begin{theorem}
\label{lemma:fsmall}
Let $f_n(s)$ be defined by \eqref{eq:diffeq}. Then, \eqref{eq:fh} holds with $c_n$ as in Table \ref{tab:cn}.
\begin{table}[H]
    \begin{tabular}{ | l | l | l | l | l | l |}
    \hline
    $n$ & $c_n$  & $n$ & $c_n $ & $n$ & $c_n $ \\  
    \hline
    $2$& $0.33$ & $9$ & $0.61$ & $21-24$ & $0.68$ \\  
    \hline
    $3$& $0.39$& $10$ & $0.61$ & $25-32$ & $0.69$\\
    \hline
    $4$& $0.45$ & $11-12$& $0.63$ & $33-46$ & $0.7$\\
    \hline
    $5$& $0.51$ & $13$ & $0.64$ & $47-80$ & $0.71$\\
    \hline
    $6$& $0.54$ & $14$ & $0.65$ & $81-308$ & $0.72$\\
    \hline
    $7$ & $0.57$ & $15-18$ & $0.66$ & $309-450$ & $0.73$\\
    \hline
    $8$& $0.58$ & $19-20$ & $0.67$ & &\\
    \hline
    \end{tabular}
    \quad
\caption{Upper bound for $c_n$}  
\label{tab:cn}   
\end{table}
\end{theorem}
We note that it is possible to
compute these bounds on $c_n$ to more decimal places and for larger $n$ if required.
We extend Theorem \ref{lemma:expf_n} to all $n$ in the
following result.
\begin{theorem}
\label{lemma:f_n}
Let $f_n(s)$ be defined by \eqref{eq:diffeq}. Then, \eqref{eq:fh} holds, for all $n$, with $c_n\le 0.9214$.
\end{theorem}
By Table \ref{tab:cn} it appears clear that the above upper bound is relatively tight. The above results will then be used in Section \ref{ELS} to prove Theorem \ref{theo:JR}, which improves on the Jurkat–Richert  version of the explicit linear
sieve \cite{JR}, proved by Nathanson in Theorem 9.7 \cite{Nathanson}. This result is of high interest as it is one of the fundamental elements that could lead to an explicit version of Chen's Theorem \ref{theo:Chen}. For a recent proof of Chen's theorem see \cite[Chapter 10]{Nathanson}, based on \cite{Iwaniec}. \hfill \break
In Section \ref{EXPS} we aim to approximate $f_n(s)$. Specifically in Subsection \ref{NS} we introduce the method, inspired by \cite{Marsaglia}, to compute $f_n(s)$ to a hight precision, in Subsection \ref{NUB} we prove Theorem \ref{lemma:fsmall} and in Subsection \ref{AA} we prove Theorem \ref{lemma:f_n}. In Section 
\ref{ELS} we prove the explicit version of the linear sieve.
Note that all the computations are performed on Sage \cite{Sage}.
\section{Numerical and analytic approximations of the solution}
\label{EXPS}
To compute $f_n(s)$ we convert the differential-difference equation \eqref{eq:diffeq} into the following integral-difference equation, with the same boundary conditions and initial value. 
\begin{lemma}
\label{intf}
Let $n\ge 2$. If $n$ is even and $s\ge 2$, or if $n$ is odd and $s\ge 3$, then
\begin{equation}
\label{eq:all}
sf_n(s)=\int_s^{\infty}f_{n-1}(t-1)dt.
\end{equation}
If $n$ is odd and $1\le s \le 3$, then
\begin{equation}
\label{eq:odd}
sf_n(s)=3f_n(3)=\int_3^{\infty}f_{n-1}(t-1)dt.
\end{equation}
\end{lemma}
This lemma makes it possible to compute $f_n(s)$ recursively. We easily obtain
\begin{equation}
\label{eq:f2}
sf_2(s)= \begin{cases} 
      s-3\log(s-1)+3 \log 3-4 & 2\le s \le 4 \\
      0 & s\ge 4,
      \end{cases} 
\end{equation}
but for larger $n$ the solution appears complicated. 
For small $n$ we can use the Sage built in function \verb|integral_numerical|, that implements the Gauss--Kronrod integration method, and Lemma \ref{intf}, to obtain a first estimate of $f_n(s)$.
This gives \eqref{eq:fh} with the following approximated values for $c_n$, for small values of $n$
\begin{equation}
\label{eq:f2}
c_2\approx 0.33 \quad c_3 \approx 0.39\quad c_4\approx 0.45\quad c_5\approx 0.51 \quad c_6\approx 0.52.
\end{equation}
There is no guarantee that these results are accurate, but they will later be useful for comparison.
Indeed, it appears that with this approach and limited computational time, we can estimate only up to $c_6$. We thus clearly need a better method to approximate $f_n(s)$.
\subsection{Numerical solution }
\label{NS}
 We can note that many ideas that appear in the papers cited in the introduction could be applied to our differential-difference equation. In this section we will focus on the idea presented in one of these papers. We draw inspiration from the similarities between $f_n(s)$ and the Dickman function, using ideas from \cite{Marsaglia}, to compute approximate solutions for $f_n(s)$.  but we will not pursue this 
 By \eqref{eq:odd} it is clear that we can focus on $s \ge 3$ when $n$ is odd. Thus for $t\ge 0$ , by \eqref{eq:all}, we obtain that
 \begin{equation}
 \label{eq:recapp}
 f_n(k+1+t)=\frac{k+1}{k+1+t}f_n(k+1)-\frac{k+1}{k+1+t}\int_{0}^{t}f_{n-1}(k+x)dx.
 \end{equation}
 This allows us to compute an approximated solution for $f_n(s)$, by induction on $n$ and $k$, using its Taylor expansion. 
 From \eqref{eq:diffeq} it is easy to see that, for $k=1,2$ and $-1 \le z \le 1$, we have 
 \begin{equation*}
 f_1(k+\frac{1}{2}+\frac{1}{2}z)=\frac{3}{k+\frac{1}{2}}-1+\sum_{n=1}^{\infty}\frac{3}{k+\frac{1}{2}}\frac{1}{(2k+1)^n}z^n,
\end{equation*} 
which will be the first step of our inductive proof. 
We now assume that we have
\begin{equation}
\label{eq:ind}
 f_{n}(k+\frac{1}{2}+\frac{1}{2}z)=b^k_{0, n}+b^k_{1, n}z+b^k_{2, n}z^2+\cdots,
 \end{equation}
for all $n \le \overline{n}$ and $k\le \overline{k}$, with $-1\le z \le 1$. The proof is now divided in two cases. First, if $\overline{n}$ is even and $\overline{k}\le \overline{n}-1$, or $\overline{n}$ odd and $k \le \overline{n}$, we seek 
 \begin{equation*}
 f_{\overline{n}}(\overline{k}+1+\frac{1}{2}+\frac{1}{2}z)=b^{\overline{k}+1}_{0, \overline{n}}+b^{\overline{k}+1}_{1, \overline{n}}z+b^{\overline{k}+1}_{2, \overline{n}}z^2+\cdots,
 \end{equation*}
 that will be the inductive step.
 By \eqref{eq:recapp} and \eqref{eq:ind}, we obtain that
 \begin{align*}
 &(\overline{k}+1+\frac{1}{2}+\frac{1}{2}z)(b^{\overline{k}+1}_{0, \overline{n}}+b^{\overline{k}+1}_{1, \overline{n}}z+b^{\overline{k}+1}_{2, \overline{n}}z^2+\cdots)=(\overline{k}+1)\cdot \\ &\cdot (b^{\overline{k}}_{0, \overline{n}}+b^{\overline{k}}_{1, \overline{n}}+b^{\overline{k}}_{2, \overline{n}}+\cdots)-\frac{1}{2}[b^{\overline{k}}_{0, \overline{n}-1}t+b^{\overline{k}}_{1, \overline{n}-1}\frac{t^2}{2}+b^{\overline{k}}_{2, \overline{n}-1}\frac{t^3}{3}+\cdots]_{-1}^z.
 \end{align*}
 That is
 \begin{align*}
 &b^{\overline{k}+1}_{0, \overline{n}}=\frac{\overline{k}+1}{\overline{k}+1+\frac{1}{2}}(b^{\overline{k}}_{0, \overline{n}}+b^{\overline{k}}_{1, \overline{n}}+\cdots)-\frac{(b^{\overline{k}}_{0, \overline{n}-1}(-1)+b^{\overline{k}}_{1, \overline{n}-1}\frac{(-1)^2}{2}+\cdots)}{2({\overline{k}}+1+\frac{1}{2})},\\
 &b^{\overline{k}+1}_{i, \overline{n}}=-\frac{1}{2}\frac{\frac{b^{\overline{k}}_{i-1, \overline{n}-1}}{i}+b^{\overline{k}+1}_{i-1, \overline{n}}}{\overline{k}+1+\frac{1}{2}} \quad \text{for} \quad i>0,
 \end{align*}
 thus concluding the inductive step in the first case.
In second case, namely if $\overline{n}$ is even and $\overline{k}= \overline{n}$, or $\overline{n}$ odd and $k=\overline{n}+1$, we seek $b^{1}_{i, \overline{n}+1}$, that will be the inductive step. We can compute these terms, using \eqref{eq:recapp} as in the first case, observing that if $\overline{n}$ is even
\begin{align*}
f_{\overline{n}+1}(1)=& \frac{1}{4}\Big( b^{\overline{n}}_{0, \overline{n}}+b^{\overline{n}}_{1, \overline{n}}\frac{1}{2}+b^{\overline{n}}_{2, \overline{n}}\frac{1}{3}+\cdots \\ & \cdots -(-b^1_{0, \overline{n}}+b^1_{1, \overline{n}}\frac{1}{2}-b^1_{2, \overline{n}}\frac{1}{3}+\cdots)\Big),
\end{align*}
and if $n$ is odd
\begin{align*}
f_{\overline{n}+1}(2)= &\frac{1}{2}\Big( b^{\overline{n}+1}_{0, n-1}+b^n_{1, \overline{n}}\frac{1}{2}+b^{\overline{n}+1}_{2, \overline{n}}\frac{1}{3}+\cdots \\& \cdots-(-b^1_{0, \overline{n}}+b^2_{1, \overline{n}}\frac{1}{2}-b^2_{2, \overline{n}}\frac{1}{3}+\cdots)\Big).
\end{align*}
This concludes the inductive step. \hfill \break 
We note that the accuracy of $f_n(s)$ depends on the accuracy of $b^{k+1}_{i, n}$. As all these constants are rational multiples of $b^{k+1}_{0, n}$, it is thus enough to store in an array, for each $n$ and $k$, the value of this last constant to a high enough precision. \hfill \break  Note that we will not use this approach in the applications as a more elementary upper bound suffice. We still think that the above framework could be of interest in different applications and thus decided to include it in this paper.
\subsection{Numerical upper bound} 
\label{NUB}
Here, we introduce a different, and more elementary, framework used to bound $f_n(n)$ by above with $h(s)$. We start observing that by \eqref{eq:odd}, for $n>2$ even and $2\le s \le 4$, we have
\begin{equation}
\label{eq:f4}
sf_n(s)=3f_{n-1}(3)\log(\frac{3}{s-1})+\int_4^{\infty}f_{n-1}(t-1)dt.
\end{equation}
Thus by \eqref{eq:f4} and \eqref{eq:odd}, we can focus on $s\ge 3$, when $n$ is odd and $s\ge 4$, when $n$ is even.
We can now note that $f_n(s)$ is decreasing in $s$ and thus we can easily bound it using Lemma \ref{intf} in the following way. Taken $0\le i \le k$ , $x_0 \le s$, $x_k\ge n+2$ and $x_{n+1}\ge x_n$, we have
\begin{equation}
sf_n(s) \le \sum_{ x_i } f(x_i)(x_i-x_{i+1}).
\end{equation}
We are thus left with optimizing on the choice of $x_i$. We opted for a uniform splitting, which appears preferable compared to, for example, an exponentially increasing one. We then used \eqref{eq:f2} as a comparison to choose $x_i-x_{i-1}=1/500$, for all $i$. Here, we report the implemented code.
\begin{lstlisting}[language=Python, caption=Upper bound $f_n$]
def f2(x):
    return 1+(-3*ln(x-1)+3*ln(3)-4)/x
m=1000
n=var('n')
j=var('j')
a=500
s_f={}
for n in range(3,a,1):
    s_f[n]=[]
    if n==3:
        for i in range((n-1)/2*m+2):
            s_f[n].append(0)                                                      
        j=(n-1)/2*m
        while 0<=j<=(n-1)/2*m:
            s_f[3][j]=N(f2(2+2*j/m)*2/m +s_f[3][j+1],                                      
            digits=10)*(1+10^(-9)) 
            j=j-1  
    if is_odd(n) and n > 3:
        for i in range((n-1)/2*m+2):
            s_f[n].append(0)   
        j=(n-1)/2*m
        while 0<=j<=(n-1)/2*m:
            s_f[n][j]=N(s_f[n-1][j]*2/m/(2+2*j/100)+
            +s_f[n][j+1], digits=10)*(1+10^(-9))
    if is_even(n):
        for i in range(n/2*m+2):
            s_f[n].append(0)
        j=(n/2-1)*m   
        while 0<=j<=(n/2-1)*m:
            s_f[n][m+j]=N(s_f[n-1][j]*2/m/(3+2*j/m)+
            +s_f[n][m+j+1], digits=10)*(1+10^(-9))
            j=j-1 
        j=(m-1)
        while 0<=j<=(m-1):
            s_f[n][j]=N(s_f[n-1][0]*(ln(3)-ln(1+2*j/m))+
            +s_f[n][m], digits=10)*(1+10^(-9))
            j=j-1                
\end{lstlisting}
In the above code, the introduction of \verb|N(, digits=10)*(1+10^(-9))| appears to be fundamental, as without it we would be able to compute an upper bound for $f_n(s)$ only up to $n=8$. Also, confronting the upper bounds for $n\le 8$ computed in the two different ways, it appears that the introduced error term does not strongly affect the accuracy of the result.
\hfill \break
We can now compute a good upper bound for $c_n$, for `small' $n$. Here, we use the fact that both $f_n(s)$ and $h(s)$ are decreasing. Indeed for each interval $[x_n, x_{n+1}]$ we have that $f_n(x_{n+1})/h(x_n)$ is an upper bound for $f_n(s)/h(s)$ for all $s \in [x_n, x_{n+1}]$. Then, to obtain the upper bound for $c_n$, we confront these values and use \eqref{eq:odd} and \eqref{eq:f4} to compute the upper bound for small $s$. Here, we report the implemented code.
\begin{lstlisting}[language=Python, caption= Upper bound for $n$ odd]
n='odd'
j=0
a=0
while 0<=j<=(n-1)/2*m:
    a=max(a,s_f[n][j]/(6*e^2*(3+2*j/m))*(3+2*(j+1)/m)*
    e^(3+2*(j+1)/m))
    j=j+1
print(max(N((s_f[n][0]/2)^(1/(n-1)), digits=10), N(a^(1/(n-1)), digits=10)))              
\end{lstlisting}
\begin{lstlisting}[language=Python, caption= Upper bound for $n$ even]
def g(f,a,b):
    (y, e) = find_local_maximum(f,a,b)
    return y
n='even'
j=m
s=var('s')
a=0
while m<=j<=(n/2)*m:
    a=max(a,s_f[n][j]/(6*e^2*(2+2*j/m))*(2+2*(j+1)/m)
    *e^(2+2*(j+1)/m) )
    j=j+1       
print(max(N(a^(1/(n-1)), digits=10), g(((s_f[n-1][0]*ln(3/(s-1))  +s_f[n][m])/(2*e^2*s)*e^s)^(1/3), 2,3), g(((s_f[n-1][0]*ln(3/(s-1))+s_f[n][m])/(6*e^2)*e^s)^(1/(n-1)), 3,4)))               
\end{lstlisting}
Note that we can expect \verb|find_local_maximum| to give an accurate result as we apply it to simple functions.
We can thus make \eqref{eq:f2} precise proving Theorem \ref{lemma:fsmall}.
We also used \verb|find_local_maximum|, and \eqref{eq:f2}, to compute the upper bound for $c_2$. 
\subsection{Analytic approximation}
\label{AA}
In this subsection we will strictly follow Nathanson's approach from Chapter 9 \cite{Nathanson}, improving his results.
We will start introducing some properties of $h(s)$.
It is easy to see that
\begin{equation*}
h(s-1)\le \frac{4e}{3}h(s)~~~\text{for} ~~~s\ge 2,
\end{equation*}
that is a more precise version of Exercise 8 \cite{Nathanson}.
For $s\ge 2$, let
\begin{equation*}
H(s)=\int_s^{\infty} h(t-1) dt.
\end{equation*}
Let
\begin{equation*}
\alpha=\frac{H(2)}{2h(2)}=\frac{e^2H(2)}{2}\approx 0.96068.
\end{equation*}
We can now give an improved version of Lemma 9.6 \cite{Nathanson}.
\begin{lemma}
We have
\begin{equation*}
H(s)\le \frac{e}{3}sh(s)~~~\text{for}~~~s\ge 3,
\end{equation*}
\begin{equation*}
H(3)\le \frac{e(2\alpha-1)}{3}sh(s)~~~\text{for}~~~2\le s\le 3,
\end{equation*}
\begin{equation*}
H(3)\le (2\alpha-1)sh(s)~~~\text{for}~~~1\le s\le 2.
\end{equation*}
\end{lemma}
\begin{proof}
If $s\ge 3$, $h(s-1)\le e^{1-s}$ and
\begin{equation*}
H(s)\le \int_s^{\infty} e^{1-t}dt=e^{1-s}=\frac{e}{3}sh(s).
\end{equation*}
If $2\le s\le 3$ we have
\begin{equation*}
H(3)=H(2)-e^{-2}=(2\alpha-1)e^{-2}\le \frac{e(2\alpha-1)}{3}sh(s).
\end{equation*}
If $1\le s\le 2$ we have
\begin{equation*}
H(3)=H(2)-e^{-2}=H(2)-h(2)=h(2)(2\alpha-1)\le (2\alpha-1)sh(s).
\end{equation*}
\end{proof}
We thus improve Lemma 9.6 \cite{Nathanson}.
\begin{theorem}
\label{theo:H}
For $\gamma =0.9214$, we have
\begin{equation*}
H(s)\le \gamma sh(s) \quad \text{for} \quad s\ge 3
\end{equation*}
and 
\begin{equation*}
H(3)\le \gamma sh(s) \quad \text{for} \quad 1\le s\le 3.
\end{equation*}
\end{theorem}
We can now use the above theorem to improve Lemma 9.7 \cite{Nathanson}, proving Theorem \ref{lemma:f_n}.
\begin{proof}
The proof is by induction on $n$. The case $n=1$ is the same as Lemma 9.7 \cite{Nathanson}. Now let $n\ge 2$, and assume the result holds for $n-1$, we can easily prove that it holds for $n$ using \eqref{lemma:f_n} and \eqref{theo:H}.
\end{proof}
Note that $\gamma$ appears optimal with the adopted framework.

\section{Explicit version of the linear sieve}
\label{ELS}
We now report some definitions and results that are fundamental to introduce the linear sieve.
Let $\mathbb{P}$ be a set of primes and $g(d):\mathbb{N}\rightarrow \mathbb{C}$ a multiplicative function. For $2\le z \le D$, with $D \in \mathbb{R}^+$, we define
\begin{equation*}
V(z):= \prod_{\substack{p\in \mathbb{P}\\ p<z}}(1-g(p)),
\end{equation*}
and
\begin{equation*}
y_m=y_m( D, p_1, \cdots,p_m)=\left( \frac{D}{p_1 \cdots p_n}\right)^{\frac{1}{2}}.
\end{equation*}
We further define
\begin{equation*}
T_n(D,z)= \sum_{\substack{p_1 \cdots p_n \in \mathbb{P}\\y_n\ge p_n < \cdots <p_1<z\\p_m<y_m\forall M<n,~m\equiv n~(\text{mod}~ 2)}}g(p_1\cdots p_n)V(p_n).
\end{equation*} 
The upper bound on $f_n(s)$ will now be fundamental as we will use it to approximate $T_n(D,z)$. 
We now need the following result that is Lemma 9.8 \cite{Nathanson}.
\begin{lemma}
Let $z\ge 2$ and $1 < w < z $. Let $\mathbb{P}$ be a set of primes, and let $g(d)$ be a multiplicative function such that
\begin{equation*}
0\le g(p) < 1 \quad \text{for all} \quad p\in \mathbb{P}
\end{equation*}
and
\begin{equation}
\label{eq:prod}
\prod_{\substack{p\in \mathbb{P}\\ u \le p < z}}(1-g(p))^{-1} \le K \frac{\log z}{\log u},
\end{equation}
for some $K> 1$ and all $u$ such that $1 < u < z$. Let 
\begin{equation*}
V(z)=\prod_{\substack{p\in \mathbb{P}\\ p < z}}(1-g(p)),
\end{equation*}
and let $\Phi $ be a continuous, increasing function on the interval $[w, z]$. Then
\begin{equation*}
\sum_{\substack{p \in \mathbb{P} \\ w \le p < z}} g(p) V(p) \Phi(p) \le (K-1)V(z) \Phi(z)-KV(z)\int_w^z \Phi(u) d\left(\frac{\log z}{\log u} \right).
\end{equation*}
\end{lemma}
We can now prove our main lemma, which is an improved version of Lemma 9.5 \cite{Nathanson}.
\begin{lemma}
\label{lemma:expf_n}
Let $z\ge 2$, and $D$ real such that 
\begin{equation*} 
s=\frac{\log D}{\log z}\ge 
\begin{cases} 
1 & \text{if}  \quad n \quad \text{is odd} \\
0 & \text{if}\quad n \quad \text{is even} .
\end{cases} 
\end{equation*}
Let $\mathbb{P}$ be a set of primes and $g(d)$ be a multiplicative function such that
\begin{equation*}
0 \le g(p) < 1 \quad \text{for all} \quad p \in \mathbb{P}
\end{equation*} 
and
\begin{equation*}
\prod_{\substack{p\in \mathbb{P}\\ u \le p < z}}(1-g(p))^{-1} \le K \frac{\log z}{\log u},
\end{equation*}
for all $u$ such that $1 < u < z$ and $k$ such that
\begin{equation*}
1< K < 1+\frac{1}{200}.
\end{equation*}
Then
\begin{equation*}
T_n(D,z)<V(z)\left( f_n(s)+(K-1)\tau_n e^2h(s)\right),
\end{equation*}
where
\begin{equation} 
\label{eq:tau}
\tau_n=
\begin{cases} 
3 & \text{if}  ~ n=1 \\
\tau_{n-1}\left(\gamma +\left(\frac{4e}{3}+\gamma\right)(K-1)+\frac{8e}{ 3}\frac{(c_{n-1})^{n-2}}{\tau_{n-1}}+2\frac{(c_n)^{n-1}}{\tau_{n-1}}\right) & \text{if}~ n \ge 2.
\end{cases} 
\end{equation}
\end{lemma}
\begin{proof}
We start defining 
\begin{equation*}
h_n(s)=(K-1)\tau_n e^2h(s).
\end{equation*}
We thus want to prove 
\begin{equation}
\label{eq:T_n}
T_n(D,z)<V(z)\left( f_n(s)+h_n(s)\right).
\end{equation}
The proof of the above result is by induction on $n$. Let $n=1$. By Lemma 9.3 \cite{Nathanson} with $\beta =2$, we have that $T_1(D,z)=0$ for $s>3$. Since the right hand side of \eqref{eq:T_n} is positive, it follows that the inequality holds for $s>3$. If $1\le s \le 3$ then $sf_1(s)=3-s$ and 
\begin{equation*}
T_1(D,z)=V(D^{1/3})-V(z)
\end{equation*}
by (9.13) \cite{Nathanson}. It easily follows 
\begin{equation*}
\frac{T_1(D,z)}{V(z)}=\left(\frac{3}{s}-1\right)+\frac{3}{s}(K-1)\le f_1(s)+h_1(s).
\end{equation*}
This proves the lemma for $n=1$. Let $n>2$ and assume that the lemma holds for $n-1$. For $n$ even and $s\ge 2$, or for $n$ odd and $s\ge 3$, we define the function
\begin{equation*}
\Phi(u)=f_{n-1}\left( \frac{\log D}{\log u}-1\right)+h_{n-1}\left( \frac{\log D}{\log u}-1\right).
\end{equation*}
Now, using the induction hypothesis for $n-1$ as done in Theorem 9.5 \cite{Nathanson} we obtain that
\begin{align*}
T_n(D,z)\le & (K-1)V(z)(f_{n-1}(s-1)+h_{n-1}(s-1))+\\&+\frac{KV(z)}{s}\int_s^{\infty} (f_{n-1}(t-1)+h_{n-1}(t-1))dt.
\end{align*}
By Theorem \ref{lemma:f_n}, we have
\begin{equation*}
\frac{K}{s}\int_s^{\infty} f_{n-1}(t-1)dt=Kf_n(s),
\end{equation*}
and, by the definition of $H(s)$ and Theorem \ref{theo:H},
\begin{equation*}
\int_s^{\infty} h(t-1)dt=H(s)\le \gamma s h(s),
\end{equation*}
and thus 
\begin{equation*}
\frac{K}{s}\int_s^{\infty} h_{n-1}(t-1)dt\le \gamma K h_{n-1}(s).
\end{equation*}
Since $h(s-1)\le \frac{4e}{3}h(s)$ for $s\ge 2$, we have
\begin{equation*}
(K-1)h_{n-1}(s-1)< \frac{4e}{3}(K-1)h_{n-1}(s),
\end{equation*}
and
\begin{align*}
&(K-1)f_{n-1}(s-1)\le (K-1)2e^2(c_{n-1})^{n-2}h(s-1)\\&\le (K-1)\frac{8e}{3}e^2(c_{n-1})^{n-2}h(s)=\frac{8e}{3 }\frac{(c_{n-1})^{n-2}}{\tau_{n-1}} h_{n-1}(s).
\end{align*}
Therefore,
\begin{equation*}
\frac{T_n(D,z)}{V(z)}<Kf_n(s)+\left(\gamma K+\frac{4e}{3}(K-1)+\frac{8e}{ 3}\frac{(c_{n-1})^{n-2}}{\tau_{n-1}}\right)h_{n-1}(s).
\end{equation*}
By the definition of $h_{n-1}(s)$, we have
\begin{equation*}
(K-1)f_n(s)\le (K-1)2e^2(c_{n})^{n-1}h(s)<2 \frac{(c_{n})^{n-1}}{\tau_{n-1}}h_{n-1}(s),
\end{equation*}
and so
\begin{equation*}
Kf_n(s)< f_n(s)+2 \frac{(c_{n})^{n-1}}{\tau_{n-1}}h_{n-1}(s).
\end{equation*}
We thus have
\begin{equation*}
\frac{T_n(D,z)}{V(z)}
\end{equation*}
\begin{equation*}
<f_n(s)+\left(\gamma +\left(\frac{4e}{3}+\gamma\right)(K-1)+\frac{8e}{ 3}\frac{(c_{n-1})^{n-2}}{\tau_{n-1}}+2\frac{(c_n)^{n-1}}{\tau_{n-1}}\right)h_{n-1}(s)
\end{equation*}
\begin{equation*}
= f_n(s)+h_{n}(s).
\end{equation*}
Let $n\ge 3$ be odd, and let $1 \le s \le 3$. If $z=D^{1/3}$, by (9.15) \cite{Nathanson} and the same argument previously used, we obtain
\begin{equation*}
T_n(D,z)<V(z)(f_n(3)+h_n(3))\le V(z)(f_n(s)+h_n(s)),
\end{equation*}
since $f_n(s)$ and $h_n(s)$ are decreasing. This completes the proof.
\end{proof}
Note that the choice of $1/200$ is the same made by Nathanson, while in our case it is possible to choose a larger value we choose not to do so here for the sake of continuity. Note that for computational purposes it is possible to take, in Lemma \ref{lemma:expf_n}, $c_1=1$.
We can now improve on Theorem 9.6 \cite{Nathanson}.
\begin{theorem}
\label{theo:G}
Let $z$, $D$, $s$, $\mathbb{P}$, $g(d)$, and $K=1+\epsilon$ satisfying the hypotheses of Lemma \ref{lemma:expf_n}. Let
\begin{equation*}
G(z, \lambda^{\pm})=\sum_{d|P(z)}\lambda^{\pm}g(d).
\end{equation*}
Then
\begin{equation*}
G(z, \lambda^{+})=V(z)\left(F(s)+\epsilon e^2 h(s) \sum_{\substack{n=1\\ n \quad \text{odd}}}^{\infty}\tau_n \right)
\end{equation*}
and 
\begin{equation*}
G(z, \lambda^{-})=V(z)\left(f(s)+\epsilon e^2 h(s) \sum_{\substack{n=1\\ n \quad \text{even}}}^{\infty}\tau_n \right),
\end{equation*}
where $F(s)$ and $f(s)$ are the continuous functions defined in (9.27) and (9.28) \cite{Nathanson}.
\end{theorem}
\begin{proof}
The proof follows easily by (9.10) \cite{Nathanson} and Lemma \ref{lemma:expf_n}.
\end{proof}
In Table \ref{tab:tau} we report upper bounds for $\tau_n$ for `small' $n$, obtained by Theorem \ref{lemma:fsmall}. Note that we use a higher decimal precision in the computations compared to Table \ref{tab:cn}. We compute an upper bound for $\tau_n$ recursively, by equation \eqref{eq:tau}, via an upper bound for the right hand side term obtained with \verb|N(, digits=10)*(1+10^(-9))|.

\begin{table}[H]
    \begin{tabular}{ | l | l | l | l | l | l |}
    \hline
    $n$ & $\tau_n $ & $n$ & $\tau_n $ & $n$ & $\tau_n $\\  
    \hline
    $1$& $3$ & $18-19$ & $8$ & $138-179$ &$10^{-2}$\\
    \hline
    $2$& $11$ & $20-22$ & $7$ & $180-220$ &$10^{-3}$\\
    \hline
    $3$& $13$ & $23-25$ & $6$ & $221-262$ & $10^{-4}$\\
    \hline
    $4-6$& $14$ & $26-29$ & $5$ & $263-303$& $10^{-5}$\\
    \hline
    $7-9$& $13$ & $30-34$ & $4$ & $304-345$ &$10^{-6}$\\
    \hline
    $10-11$& $12$ & $35-42$ & $3$ & $346-386$ & $10^{-7}$\\
    \hline
   $12$ & $11$ & $43-54$ & $2$ & $387-428$ & $10^{-8}$\\
    \hline
    $13-15$ & $10$ & $55-96$ & $1$ & $429-450$& $10^{-9}$\\
    \hline
    $16-17$&$9$ & $97-137$& $10^{-1}$ &&\\
    \hline
    \end{tabular}
\caption{Upper bound for $\tau_n$}  
\label{tab:tau}  
\end{table}
We now obtain a bridging result.
\begin{lemma}
\label{lemma:exprec}
We have 
\begin{equation*}
F_1=\sum_{\substack{n=1\\ n \quad \text{odd}}}^{\infty}\tau_n \le 164
\end{equation*}
and 
\begin{equation*}
f_1=\sum_{\substack{n=1\\ n \quad \text{even}}}^{\infty}\tau_n\le 162.
\end{equation*}
\end{lemma}
\begin{proof}
By the definition of $\tau_n$ it is easy to see that for any $k_1$ odd
\begin{equation*}
F_1\le \sum_{\substack{n=1\\ n \quad \text{odd}}}^{k_1}\tau_{n} +\tau_{k_1+1} \sum_{\substack{n=1\\ n \quad \text{odd}}}^{\infty}\left( \frac{\tau_{k_1+2}}{\tau_{k_1+1}}\right)^n
\end{equation*}
 and for any $k_2$ even
 \begin{equation*}
f_1\le\sum_{\substack{n=1\\ n \quad \text{even}}}^{k_2}\tau_{n} +\tau_{k_2+1} \sum_{\substack{n=1\\ n \quad \text{odd}}}^{\infty}\left( \frac{\tau_{k_2+2}}{\tau_{k_2+1}}\right)^n.
\end{equation*}
The result now follows from Table \ref{tab:tau}, note that we used a higher precision in the computations.
\end{proof}
We can now improve Theorem 9.7 \cite{Nathanson}. The improvement will be on the constant $e^{14}$ that will be reduced to around $160\cdot e^2$.
\begin{theorem}[Jurkat--Richert]
\label{theo:JR}
Let $A=\{a(n)\}_{n=1}^{\infty}$ be an arithmetic function such that
\begin{equation*}
a(n)\ge 0 \quad \text{for all} \quad n \quad \text{and} \quad 
|A|=\sum_{n=1}^{\infty} a(n)< \infty.
\end{equation*}
Let $\mathbb{P}$ be a set of prime numbers and for $z\ge 2$, let
\begin{equation*}
P(z)=\prod_{\substack{p \in \mathbb{P}\\ p<z}}p.
\end{equation*}
Let 
\begin{equation*}
S(A,\mathbb{P}, z)=\sum_{\substack{n=1\\ (n, P(z))=1}}^{\infty}a(n).
\end{equation*}
For every $n\ge 1$, let $g_n(d)$ be a multiplicative function such that
\begin{equation*}
0\le g_n(p)< 1 \quad \text{for all} \quad p \in \mathbb{P}.
\end{equation*}
Define $r(d)$ by
\begin{equation*}
|A_d|=\sum_{\substack{n=1\\ d|n}}^{\infty} a(n)= \sum_{n=1}^{\infty} a(n)g_n(d)+r(d).
\end{equation*}
Let $\mathbb{Q}\subseteq \mathbb{P}$, and $Q$ the products of its primes. Suppose that, for some $\epsilon$ such that $0<\epsilon<1/200$, the inequality
\begin{equation*}
\prod_{\substack{p\in \mathbb{P}/\mathbb{Q}\\ u \le p < z}}(1-g(p))^{-1} \le (1+\epsilon) \frac{\log z}{\log u},
\end{equation*}
holds for all $n$ and $1<u<z$. Then, for any $D\ge z$ there is an upper bound
\begin{equation*}
S(A,\mathbb{P},z)<(F(s)+\epsilon 164 e^2h(s))X+R,
\end{equation*}
and for any $D\ge z^2$ there is a lower bound
\begin{equation*}
S(A,\mathbb{P},z)<(f(s)-\epsilon 162 e^2h(s))X+R,
\end{equation*}
where
\begin{equation*}
s=\frac{\log D}{\log z},
\end{equation*}
$f(s)$ and $F(s)$ are two functions defined in (9.27) and (9.28) \cite{Nathanson},
\begin{equation*}
X=\sum_{n=1}^{\infty} a(n) \prod_{p|P(z)}(1-g_n(p)),
\end{equation*}
and the remainder term is
\begin{equation*}
R=\sum_{\substack{d|P(z)\\ d<QD}}|r(d)|.
\end{equation*}
If there is a multiplicative function $g(d)$ such that $G_n(d)=g(d)$ for all $n$, then
\begin{equation*}
X=V(z)|A|, \quad \text{where} \quad 
V(z)=\prod_{p|P(z)}(1-g_n(p)).
\end{equation*}
\end{theorem}
\begin{proof}
The proof is the same as the proof of Theorem 9.7 \cite{Nathanson}, using Theorem \ref{theo:G} and Lemma \ref{lemma:exprec}.
\end{proof}
We can observe that 
\begin{equation*}
164 e^2h(s) \epsilon \le 0.82 \quad \text{and} \quad 
 162 e^2h(s)\epsilon \le 0.81,
\end{equation*}
this gives a uniform upper bound for the 'constants' appearing in Theorem \ref{theo:JR}. In \cite{Nathanson} the constant was approximately $\approx 810$, and thus our is around $1000$ times smaller. \hfill \newline 
As hinted before, the choice of $\epsilon=1/200$ is made to be consistent with the one in Nathanson's result, where it is the biggest number such that $f_1$ and $F_1$ converge. In our result we can take $\epsilon$ significantly larger. By the definition of $\tau_n$, it is clear that we need
\begin{equation*}
\gamma +\left(\frac{4e}{3}+\gamma\right)(K-1)<1,
\end{equation*}
that holds for $\epsilon \ge 1/63$. We now conclude the paper reporting the upper bounds for $f_1$ and $F_1$ for different $\epsilon$.
\begin{table}[H]
    \begin{tabular}{ | l | l | l | l | l | l | l | l | l |}
    \hline
    $\epsilon^{-1}$ & $f_1 $ & $F_1$ &$\epsilon^{-1}$ & $f_1 $ & $F_1$&$\epsilon^{-1}$ & $f_1 $ & $F_1$ \\  
    \hline
    $63$& $32881$&$32875$ & $72$& $865$&$867$ & $81$& $500$&$500$\\
    \hline
    $64$& $7582$&$7580$ &$73$& $790$&$791$ & $84$& $450$&$450$\\
    \hline
    $65$& $3890$& $3890$ & $74$& $729$& $730$ & $87$& $400$&$400$\\
    \hline
    $66$& $2542$& $2542$ & $75$& $678$& $679$ & $93$& $350$&$350$\\
    \hline
    $67$& $1880$& $1881$ & $76$& $635$& $636$ & $99$& $300$& $300$\\
    \hline
    $68$& $1480$&$1500$ & $77$& $598$&$600$ & $114$& $250$& $250$\\
    \hline
   $69$ & $1254$&$1255$ & $78$ & $566$&$568$ & $143$& $200$& $200$\\
    \hline
    $70$&$1084$&$1086$ & $79$& $538$&$540$ &  $249$&$150$&$150$\\
    \hline
    $71$&$960$&$962$ & $80$&$514$&$515$ &&&\\
    \hline
    \end{tabular}
    
\caption{Upper bounds for $F_1$ and $f_1$ for certain $\epsilon$}  
\label{tab:fF}  
\end{table}
\section*{Acknowledgements}
I would like to thank my supervisor Tim Trudgian for his help in developing this paper and his insightful comments. I would also like to thank Pieter Moree for his helpful comments.

\end{document}